\documentclass[10pt]{amsart}
\usepackage{amssymb,enumerate}
\usepackage{amsmath}
\usepackage{graphics}

\usepackage{amsthm}

\theoremstyle{lema}

\theoremstyle{proposition}

\theoremstyle{theorem}
\newtheorem{theorem}{Theorem}[section]

\theoremstyle{theorem}

\theoremstyle{corollary}
\newtheorem{corollary}{Corolarry}[section]

\theoremstyle{definition}
\newtheorem{definition}{Definition}[section]

\theoremstyle{lemma}
\newtheorem{lemma}{Lemma}[section]

\theoremstyle{example}
\newtheorem{example}{Example}[section]

\theoremstyle{claim}

\begin{document}

\title{Some remarks on expansive mappings in metric spaces}
\author{Ovidiu Popescu}
\address{Ovidiu Popescu \newline
\indent Transilvania University if Bra\c sov \newline
\indent Department of Mathematics and Computer Sciences \newline
\indent Iuliu Maniu 50, Bra\c sov, Romania}
\email{ovidiu.popescu@unitbv.ro}

\author{Cristina Maria P\u acurar}
\address{Cristina Maria P\u acurar \newline
	\indent Transilvania University if Bra\c sov \newline
	\indent Department of Mathematics and Computer Sciences \newline
	\indent Iuliu Maniu 50, Bra\c sov, Romania}
\email{cristina.pacurar@unitbv.ro}

\begin{abstract}
 The aim of this paper is to generalize the results on expansive mappings  of Ye\c silkaya and Aydin from \cite{Yesilkaya}. We give some fixed point results for q-expansive mappings in metric spaces and prove some fixed point theorems for this class of mappings. Finally, we present some examples to support the new results.
\end{abstract}

\maketitle
\pagestyle{myheadings}
\markboth{Popescu O., P\u acurar C.}{Some remarks in expansive mappings in metric spaces}

\section{Introduction and preliminaries}

In 1984, Wang et. al. \cite{Wang} started the study of expansive mappings and proved some fixed point theorems for such mappings, which correspond to some contractive mappings in metric spaces. Thereafter, several authors generalised and extended the results of Wang, see \cite{Daffer},  \cite{Gornicki}, \cite{Gurban}, \cite{Huang}, \cite{Khan}, \cite{Rhoades}, \cite{Taniguchi}. Recently, Ye\c silkaya and Aydin (see \cite{Yesilkaya}) introduced the concept of $\theta$-expansive mappings in ordered metric spaces and extended the main results for expansive mappings from the current literature. For example, they obtained a common  fixed point theorem of two weekly compatible mappings in metric spaces.

In 1982, Sessa \cite{Sessa} defined the concept of weak commutativity for two mappings and proved a common fixed  point theorem for such mappings. In 1986, Jungck \cite{Jungck} introduced the concept of weakly compatible mappings.

\begin{definition}\cite{Sessa}
	Let $U$ and $V$ be self mappings of a set $M$. A point $x \in M$ is called a coincidence point of $U$ and $V$ if and only if $Uz=Vz$. In this case, $w=Uz=Vz$ is called a coincidence of $U$ and $V$.
\end{definition}

\begin{definition}\cite{Jungck}
	Two self mappings $U$ and $V$ of a metric space $(M,d)$ are said to be weakly compatible if and only if at every point $z \in M$ which is a coincidence point of $U$ and $V$, the operators commute, that is $UVz=VUz$.
\end{definition}

\begin{theorem}{\normalfont \cite{Wang}}
	Let $(M,d)$ be a complete metric space and $U$ a self mapping of $M$. If $U$ is surjective and satisfies 
	$$d(Ux,Uz)\geq q d(x,z),$$
	for all $x,z\in M$, with $q >1$, then $U$ has a unique fixed point in $M$.
\end{theorem}

In 2004, Ran and Reurings \cite{Ran} proved a fixed point theorem in a partially ordered metric space.

\begin{theorem} \cite{Ran}
	Let $(M,\leq)$ be an ordered set and $d$ be a metric on $M$ such that $(M,d)$ is a complete metric space. Let $U:M \to M$ be a nondecreasing mapping, i.e. $Ux \leq Uy$, for every $x,y \in M$ with $x\leq y$. Suppose that there exists $x_0\in M$ with $x_0 \leq Ax_0$ and $L \in [0,1)$ such that 
	$$d(Ux,Uy) \leq Ld(x,y),$$
	for every $x,y \in M$ with $x\leq y$.
	If $U$ is continuous, then it has a fixed point in $M$.		 
\end{theorem}

Thereafter, many authors considered the problem of the existence of a fixed point for contraction type operators on partially ordered set, see \cite{Abbas}, \cite{Agarwal}, \cite{Durmaz}, \cite{Kumam}, \cite{Minak}. In 2014, Jleli and Samet introduced in \cite{Jleli} the class of $\theta$-contractions. They considered $\Theta$, the set of functions $\theta : (0,\infty)\to (1,\infty)$ satisfying the following conditions:

\begin{itemize}
	\item[($\theta_1$)] $\theta$ is non-decreasing;
	\item[($\theta_2$)] for each sequence $\{t_n\} \in (0, \infty)$, $\lim\limits_{n\to \infty}\theta(t_n) = 1$ if and only if $\lim\limits_{n\to \infty} t_n = 0^+$;
	\item[($\theta_3$)] there exists $r \in (0,1)$ and $l\in(0,\infty]$ such that 
	$$ \lim\limits_{t\to 0^+} \dfrac{\theta(t)-1}{t^r}=l.$$
\end{itemize}

The following lemma, proved by Gornicki in \cite{Gornicki} is an important tool in this theory.

\begin{lemma}{\normalfont \cite{Gornicki}}
	Let $(M,d)$  be a metric space and $U : M \to M$ a surjective mapping. Then, $U $ has a right inverse mapping, i.e., there exists a mapping $U^*:M\to M$ such that $U\circ U^*=I_M$.
\end{lemma}

If $(M, \leq)$ is an ordered set and $d$ is a metric on $M$, we say that $(M, \leq , d)$ is an ordered metric space. If for every increasing sequence $\{x_n\} \subseteq M$ with $x_n \to x^* \in M$ we have $x_n \leq x^*$ for all $n \in \mathbb{N}$, then we say that $M$ is a regular ordered metric space.

Very recently, Ye\c silkaya and Aydin introduced in \cite{Yesilkaya} the notion of $\theta$-expansive mapping in ordered metric spaces.

\begin{definition} \cite{Yesilkaya}
	Let $(M,d)$  be an ordered metric space. A mapping $U : M \to M$ is said to be $\theta$-expansive if there exists $\theta \in \Theta$ and $\eta > 1$ such that 
	$$\theta(d(Ux,Uz)) \geq [\theta(d(x,z))]^{\eta},$$
	for all $(x,z) \in M_0$, where 
	$$M_0 = \{(x,z) \in M\times M \; : \; x \leq z, \, d(Ux, Uz) > 0\}.$$
\end{definition}

They proved the following theorems:

\begin{theorem}\cite{Yesilkaya}
	Let $(M, \leq, d)$ be an ordered complete metric space, $U : M \to M$ a surjective $\theta$-expansive mapping and $U^*$ a right inverse of $U$ such that $U^*$ is $\leq$ increasing. Suppose that there exists $x_0 \in M$ such that $x_0 \leq U^*x_0$. If $U$ is continuous or $M$ is regular,  then $U$ has a fixed point. 
	\label{Theorem1.3}
\end{theorem}

\begin{theorem}\cite{Yesilkaya}
	Let $(M,d)$ be a complete metric space and $U : M \to M$ a continuous surjective $\theta$-expansive mapping. If there exists $\eta > 1$ such that 
	$$\theta(d(Ux,Uz)) \geq [\theta(\min\{d(x,z), d(x,Ux), d(z, Uz)\})]^{\eta},$$
	for all $x,z \in M$, then $U$ has a fixed point.
	\label{Theorem1.4}
\end{theorem}

We show that in these theorems it is not necessary that $\theta \in \Theta$. We can prove that the results are available even if $\theta$ has only the property that it is a non-increasing function. Thus, our results are much more less restrictive and consequently, more general than the results existing in literature.

In \cite{Yesilkaya}, there is provided the following common fixed point theorem for weakly compatible mappings (see Theorem 5).

\begin{theorem}\cite{Yesilkaya}
	Let $(M,d)$  be a complete metric space. Let $U$ and $V$ be weakly compatible self mappings of $M$ and $V(M) \subseteq U(M)$. Suppose that $\theta \in \Theta$ and there exists a constant $\eta > 1$ such that 
	$$\theta(d(Ux,Uz)) \geq [\theta(d(Vx,Vz))]^{\eta},$$
	for all $x,z\in M$. If one of the subspaces $U(M)$ or $V(M)$ is complete, then $U$ and $V$ have a unique common fixed point in $M$.
	\label{WronhThAB}
\end{theorem}

In this paper we give a more general equivalent of Theorem \ref{WronhThAB} and thus, we provide a new significant common fixed point result for weakly compatible mappings.

An essential tool in the proofs of our results is the following Lemma proved by Popescu in \cite{Popescu}:

\begin{lemma} \cite{Popescu}
	Let $(X,d)$ be a metric space and $\{x_n\}$ be a sequence in $X$ which is not Cauchy and $\lim\limits_{n\to \infty} d(x_n,x_{n+1}) = 0$. Then there exists $\varepsilon > 0$ and two sequences $\{x_{n_k}\}$ and $\{x_{m_k}\}$ of $\{x_n\}$ such that 
	$$ \lim\limits_{k\to \infty}  d(x_{n_k+1}, x_{m_k+1}) = \lim\limits_{k\to \infty}d(x_{n_k}, x_{m_k}) = \varepsilon^+.$$
	\label{Popescu}
\end{lemma}

\section{Main results}

First, let us start with the definition of $\varphi$-expansive mappings in ordered metric spaces.

\begin{definition}
	Let $(M,\,\leq,\, d)$ be an ordered metric space. A mapping $U:M\to M$ is said to be $\varphi$-expansive if there exist a non-decreasing function $\varphi: (0,\infty) \to (1,\infty)$ and $\eta > 1$ such that 
	$$\varphi(d(Ux,Uz)) \geq [\varphi(d(x,z))]^{\eta},$$
	for all $(x,z) \in M_0$, where
	$$M_0 = \{(x,z) \in M\times M \, : \, x\leq z,\, d(Ux,Uz) > 0\}.$$		
\end{definition}

The first result is a generalization of Theorem 1.3.

\begin{theorem}
	Let $(M,\,\leq,\,d)$ be an ordered complete metric space, $U : M \to M$ be a surjective $\varphi$-expansive mapping and $U^*$ a right inverse of $U$ such that $U^*$ is $\leq$ increasing. Suppose that there exists $x_0 \in M$ such that $x_0 \leq U^*x_0$. If $U$ is continuous or $M$ is regular,  then $U$ has a fixed point.
	\label{Teorema1} 
\end{theorem}

\begin{proof}
	Let $x_0 \in M$ with $x_0 \leq U^*x_0$. We define the sequence $\{x_n\}$ by $x_{n+1} = U^*x_n$. Then, we have 
	$$Ux_{n+1} = UU^*x_n = x_n,$$ 
	for all $n=0,1,2,\dots$
	
	Since $x_0 \leq U^*x_0 = x_1$ and $U^*$ is increasing, we get 
	$$U^*x_0 \leq U^*x_1,$$
	i.e., $x_1 \leq x_2$.  If there exists $n \in \mathbb{N}$ such that $x_n = x_{n+1}$, then $x_{n+1}$ is a fixed point of $U$.
	
	Now assume that $x_n \neq x_{n+1}$, for all $n \in \mathbb{N}$. Inductively, by $x_n \leq x_{n+1}$ we obtain $U^*x_n \leq U^* x_{n+1}$, i.e., $x_{n+1}\leq x_{n+2}$, so 
	$$x_0 \leq x_1 \leq x_2 \leq \dots \leq x_n \leq \dots$$
	
	Let $s= \dfrac{1}{\eta}$. Since $\eta > 1$, we have $s <1$. Since $d(Ux_{n},Ux_{n+1}) = d(x_{n-1},x_{n}) > 0$ and $x_{n} \leq x_{n+1}$ for all $n \in \mathbb{N}$, then $(x_{n},x_{n+1}) \in M_0$. So, we have for all $n \in \mathbb{N}$
	$$\varphi(d(x_{n-1}, x_n)) = \varphi(d(Ux_n,Ux_{n+1})) \geq [\varphi(d(x_{n},x_{n+1}))]^{\eta}$$
	by where
	$$\varphi(d(x_{n}, x_{n+1})) \leq [\varphi(d(x_{n-1},x_{n}))]^{s} < \varphi(d(x_{n-1},x_n)).$$
	
	Since $\varphi$ is a non-decreasing function, we get $$d(x_n,x_{n+1}) < d(x_{n-1},x_n)$$ for all $n \in \mathbb{N}$, hence $\{d(x_n,x_{n+1})\}_{n\geq 0}$ is a decreasing sequence of positive numbers. Therefore, $\{d(x_n,x_{n+1})\}_{n\geq 0}$ converges to some $d \geq 0$.
	
	Suppose $d>0$. Then, we have
	$$\varphi(d^+) \leq \varphi(d(x_n,x_{n+1}))$$
	and
	$$\varphi(d^+) \leq [\varphi(d(x_{n-1},x_{n}))]^s,$$
	for all $n \in \mathbb{N}$.
	
	Since $\varphi$ is non-decreasing, letting $n$ tend to $\infty$, we obtain
	$$1<\varphi(d) \leq \varphi(d^+) \leq \varphi(d^+)^s,$$
	which is a contradiction. Therefore, $d=0$.
	
	Now, we suppose that $\{x_n\}$ is not a Cauchy sequence. Then, by Lemma \ref{Popescu}, there exist $\varepsilon > 0$ and two subsequences $\{x_{n(k)}\}$, $\{x_{m(k)}\}$ of $\{x_n\}$ with $n(k) > m(k) \geq k$ such that 
	$$d(x_{n(k)},x_{m(k)}) \to \varepsilon^+, \quad d(x_{n(k)+1},x_{m(k)+1}) \to \varepsilon^+.$$
	
	Since $(x_{m(k)+1}, x_{n(k)+1})\in M_0$, we have 
	\begin{align*}
		\varphi(d(x_{n(k)},x_{m(k)})) & = \varphi(d(Ux_{n(k)+1},Ux_{m(k)+1}))   \\
		& \geq  [\varphi(d(x_{n(k)+1},x_{m(k)+1}))]^{\eta},
	\end{align*}
	for every $k \geq 1$.
	
	Letting $k \to \infty$, we obtain 
	$$\varphi(\varepsilon^+) \geq \varphi(\varepsilon^+)^\eta \geq \varphi(\varepsilon) > 1,$$
	which is a contradiction. Therefore, $\{x_n\}$ is a Cauchy sequence. Since $(M,d)$ is complete, we get that there exists $x^* \in M$ such that $\lim\limits_{n\to \infty}x_n = x^*$.
	
	Now we shall show that $x^*$ is a fixed point of $U$. If $U$ is continuous, then we have
	$$ x^* = \lim\limits_{n\to \infty} x_n = \lim\limits_{n\to \infty} Ux_{n+1} = U(\lim\limits_{n\to \infty}x_{n+1}) = Ux^*,$$
	i.e., $x^*$ is a fixed point of $U$. If $M$ is regular, than $x_n \leq x^*$ for all $n \in \mathbb{N}$. If there exists $n_0 \in \mathbb{N}$ such that $x_{n_0} = x^*$, than we have 
	$$U^*x^* = U^*x_{n_0}=x_{n_0+1}\leq x^*$$
	and 
	$$x^* = x_{n_0} \leq x_{n_0+1}= U^* x_{n_0} = U^*x^*.$$
	
	Hence $U^*x^* = x^*$. Otherwise, we have $x_n \neq x^*$ for every $n \in \mathbb{N}$. If $U^*x^* \neq x^*$, we have 
	$$	\varphi(d(x_{n},x^*))  = \varphi(d(Ux_{n+1},UU^*x^*)) $$
	$$\geq  [\varphi(d(x_{n+1},U^*x^*))]^{\eta} \geq \varphi(d(x_{n+1},U^*x^*)).$$
	
	Then, we get $$d(x_n,x^*) \geq d(x_{n+1}, U^*x^*).$$
	
	Letting $n$ tend to $\infty$, we obtain $d(x^*,U^*x^*) \leq 0$, which is a contradiction.
	
	Thus, we conclude that $d(x^*,U^*x^*) = 0$, that is $U^*x^* = x^*$ and $UU^*x^*=Ux^*$. Therefore, $x^*=Ux^*$.
\end{proof}

\begin{corollary}
	Let $(M,d)$ be a complete metric space and $U : M \to M$ be a continuous surjective $\varphi$-expansive mapping. If there exists $\eta >1$ such that 
	$$ \varphi(d(Ux,Uz)) \geq [\varphi(d(x,z))]^{\eta} $$
	for all $x,z \in M$, then $U$ has a unique fixed point in $M$. 
\end{corollary}

\begin{proof}
	By Theorem \ref{Teorema1} we have that $U$ has a fixed point $x^* \in M$. Suppose that $y^* \in M$ is another fixed point of $U$. Then
	$$\varphi(d(x^*,y^*)) = \varphi(d(Ux^*,Uy^*)) \geq [\varphi(d(x^*,y^*))]^{\eta}.$$
	
	Since $\varphi(d(x^*,y^*)) > 1$ and $\eta >1$, this is a contradiction. Thus, $U$ has a unique fixed point.
\end{proof}

We provide an example to illustrate our results. 

\begin{example}
	Let $Y =\left\{\dfrac{1}{r+1}, \, r \in \mathbb{N} \cup \{0\}\right\} \cup \{0\}$ endowed with the metric
	$$d\left(\frac1r,\frac1{r+p}\right) = \frac1r,$$
	for every $r, p \in \mathbb{N}$. Let us consider the order relation $\preccurlyeq$ on $Y$ defined as
	\begin{equation*}
		x\preccurlyeq z \iff x = z \text{ or } x < z < 1.
	\end{equation*}
	where $\leq$ is the usual order. 
	
	Then $(Y,\preccurlyeq,d)$ is an ordered complete metric space.
	
	Let $U : Y \to Y$ be defined as
	\[Ux=\begin{cases}
		\dfrac 1r, \quad x = \dfrac1{r+1}, \; r\in \mathbb{N} \\
		0,\,\quad x =0\\
		1, \,\quad x=1
	\end{cases}\]

	Taking \[U^*x=\begin{cases}
		\dfrac 1{r+1}, \,\quad x = \dfrac1{r}, \; r\in \mathbb{N} \\
		0, \qquad \quad x =0,
	\end{cases}\]
	
	clearly $U^*$ is $\preccurlyeq$-increasing. 
	
	Let $\varphi(t) = e^{e^{-\frac1t}}$ for every $t >0$ and let $1< \eta < e$ Then, it is true that 
	\begin{equation*}
		e^{e^{-\frac1{d(Ux,Uz)}}} \geq e^{\eta e^{-\frac1{d(x,z)}}},
	\end{equation*}
	for every $x,z \in Y$ with $x \preccurlyeq z$, since $e^{-r} \geq \eta e^{-(r+1)}$ for every $r \in \mathbb{N}$.
	
	Thus, $U$ satisfies the hypothesis of Theorem \ref{Teorema1} which implies that it has a unique fixed point in $Y$.
	
	The operator $U$ is not an expansive operator in metric spaces since
	\[\lim\limits_{r\to \infty} \dfrac{d(Ax,Az)}{d(x,z)} = \dfrac{r+1}r = 1.\]
\end{example}

The following Theorem is a generalization of Theorem \ref{Theorem1.4}.

\begin{theorem}
	Let $(M,d)$ be a complete metric space and $U : M \to M$ a continuous surjective $\varphi$-expansive mapping. If there exists $\eta >1$ such that 
	\begin{equation} 
		\varphi(d(Ux,Uz)) \geq [\varphi(\min\{d(x,z), d(x, Ux), d(z, Uz)\})]^{\eta} 
		\label{*}
	\end{equation}
	for all $x,z \in M \setminus \{t \in M: Ut=t\}$ with $Ux\neq Uz$ then $U$ has a fixed point. 
\end{theorem}

\begin{proof}
	Let $x_0$ be an arbitrary point in $M$. Since $U$ is surjective, there exists $x_1 \in M$ such that $x_0 = Ux_1$. In general, if $x_n \in M$, we can choose $x_{n+1} \in M$ such that $x_n = Ux_{n+1}$, for all $n = 0,1,2,\dots$ . If there exists $n \in \mathbb{N}$ such that $x_n = x_{n+1}$, then $x_n$ is a fixed point of $U$. Otherwise, we have $x_n \neq x_{n+1}$, for all $n \in \mathbb{N}$. Then, from equation (\ref{*}), for $x=x_n$ and $z= x_{n+1}$, we have
	\begin{align*}
		\varphi(d(x_{n-1},x_n)) &= \varphi(d(Ux_n,Ux_{n+1})) \\
		&\geq [\varphi(\min\{d(x_n,x_{n+1}), d(x_n, Ux_{n}), d(x_{n+1}, Ux_{n+1})\})]^{\eta},
	\end{align*}
	where $\min\{d(x_n,x_{n+1}), d(x_n, Ux_{n}), d(x_{n+1}, Ux_{n+1})\} = \min\{d(x_n,x_{n+1}), d(x_n, x_{n-1})\}.$
	
	If $d(x_{n-1},x_{n}) \leq d(x_n, x_{n+1})$, then we get 
	$$\varphi(d(x_{n-1},x_{n})) \geq [\varphi(d(x_{n-1},x_n))]^{\eta},$$
	which is a contradiction. Therefore $d(x_{n-1},x_{n}) > d(x_n, x_{n+1})$, for all $n \in \mathbb{N}$. Then, we obtain 
	\[\varphi(d(x_{n-1},x_n)) \geq [\varphi(d(x_{n},x_{n+1}))]^{\eta}.\]
	
	Since $\{d(x_n,x_{n+1})\}_{n\geq 0}$ is a decreasing sequence of positive numbers we get that there exists $d \geq 0$ such that $\lim\limits_{n\to \infty}d(x_n,x_{n+1}) = d$.
	
	Suppose $d>0$. Then, letting $n$ tend to $\infty$ in the above equation we obtain
	$$\varphi(d^+) \geq [\varphi(d^+)]^{\eta},$$
	which is a contradiction. Therefore, $d=0$.
	
	Now, we suppose that $\{x_n\}$ is not a Cauchy sequence. Then, by Lemma \ref{Popescu}, there exist $\varepsilon > 0$ and two subsequences $\{x_{n(k)}\}$, $\{x_{m(k)}\}$ of $\{x_n\}$ with $n(k) > m(k) \geq k$ such that 
	$$d(x_{n(k)},x_{m(k)}) \to \varepsilon^+, \quad d(x_{n(k)+1},x_{m(k)+1}) \to \varepsilon^+.$$
	
	Taking $x= x_{n(k)+1}$ and $z=x_{m(k)+1}$ in equation (\ref{*}) we obtain
	\begin{align*}
		&\varphi(d(x_{n(k)},x_{m(k)}))  = \varphi(d(Ux_{n(k)+1},Ux_{m(k)+1}))  \geq \\
		& \geq  [\varphi(\min\{d(x_{n(k)+1},x_{m(k)+1}), d(x_{n(k)+1},Ux_{n(k)+1}), d(x_{m(k)+1},Ux_{m(k)+1})\})]^{\eta},
	\end{align*}
	where\begin{align*}
		&\min\{d(x_{n(k)+1},x_{m(k)+1}), d(x_{n(k)+1},Ux_{n(k)+1}), d(x_{m(k)+1},Ux_{m(k)+1})\} =\\ &=\min\{d(x_{n(k)+1},x_{m(k)+1}), d(x_{n(k)+1},x_{n(k)}), d(x_{m(k)+1},x_{m(k)})\}.
	\end{align*}
	
	Since $d(x_n,x_{n+1}) \to 0$ as $n \to \infty$, for $k$ large enough we have that $d(x_{n(k)+1},x_{n(k)}) < \varepsilon$ and $d(x_{m(k)+1},x_{m(k)}) < \varepsilon$, hence 
	\begin{align*}
		\min\{d(x_{n(k)+1},x_{m(k)+1}), d(x_{n(k)+1},Ux_{n(k)+1}), d(x_{m(k)+1},Ux_{m(k)+1})\}= \\=\min\{d(x_{n(k)+1},x_{m(k)+1})\}.
	\end{align*}
	
	This implies that 
	\[ \varphi(d(x_{n(k)},x_{m(k)})) \geq [\varphi(d(x_{n(k)+1},x_{m(k)+1}))]^{\eta}, \]
	so letting $k \to \infty$, we get 
	
	$$\varphi(\varepsilon^+) \geq \varphi(\varepsilon^+)^\eta,$$
	which is a contradiction. Therefore, $\{x_n\}$ is a Cauchy sequence and there exists $x^* \in M$ such that $\lim\limits_{n\to \infty}x_n = x^*$.
	
	Since $U$ is continuous, we have
	
	$$ x^* = \lim\limits_{n\to \infty} x_n = \lim\limits Ux_{n+1} = U(\lim\limits_{n\to \infty}x_{n+1}) = Ux^*,$$
	hence $x^*$ is a fixed point of $U$.
\end{proof}

\begin{theorem}
	Let $(M,d)$ be a complete metric space. Let $U$ and $V$ be weakly compatible self mappings of $M$ and $V(M) \subseteq U(M)$. Suppose that $\varphi$ is a non-decreasing function $\varphi : (0,\infty) \to (1,\infty)$ and there exists a constant $\eta >1$ such that  
	\begin{equation} 
		\varphi(d(Ux,Uz)) \geq [\varphi(d(Vx,Vz))]^{\eta} 
		\label{**}
	\end{equation}
	for all $x,z \in M$ with $Vx\neq Vz$. If one of the subspace $U(M)$ or $V(M)$ is complete, then $U$ and $V$ have a unique common fixed point in $M$.
	\label{Th3}
\end{theorem}

\begin{proof}
	Let $x_0$ be an arbitrary point in $M$. Since $V(M) \subseteq U(M)$, choose $x_1\in M$ such that $y_1 = Ux_1 =Vx_0$. In general, for $x_n \in M$ we can choose $x_{n+1} \in M$ such that $y_{n+1} = Ux_{n+1} = Vx_n$. 
	
	If $y_n = y_{n+1}$, then we have 
	$$ y_n = Ux_n = Vx_{n-1} = Ux_{n+1} = Vx_n = y_{n+1}. $$
	
	Since $Ux_n = Vx_n$, $x_n$ is a coincidence of $U$ and $V$, so the weak compatibility of $U$ and $V$ ensures that
	$$UVx_n = VUx_n = UUx_n = VVx_n.$$
	
	Then, we have two possibilities. If $Vx_n \neq VVx_n$, then from (\ref{**}) we get 
	$$\varphi(d(Ux_n,UVx_n)) \geq [\varphi(d(Vx_n, VVx_n))]^\eta,$$
	but since $Ux_n = Vx_n$ and $UVx_n = VVx_n$, the above inequality becomes
	$$\varphi(d(Vx_n,VVx_n)) \geq [\varphi(d(Vx_n, VVx_n))]^\eta,$$
	which is a contradiction.
	
	On the other hand, if $Vx_n = VVx_n$, then we obtain 
	$$Vx_n = VVx_n = UVx_n,$$
	and thus $Vx_n$ is a common fixed point of $U$ and $V$.
	
	Now, suppose that $y_n \neq y_{n+1}$, for all $n \in \mathbb{N}$. Let $s= \dfrac{1}{\eta}$. Since $\eta > 1$, we get $s < 1$. Then, from (\ref{**}) for $x = y_{n+1}$, $z=y_{n+2}$, we obtain 
	$$\varphi(d(y_{n+1}, y_{n+2})) = \varphi(d(Vx_n, Vx_{n+1})) \leq  [\varphi(d(Ux_n,Ux_{n+1}))]^s$$
	$$ = [\varphi(d(Vx_{n-1},Vx_{n}))]^s = [\varphi(d(y_n,y_{n+1}))]^s.$$
	
	Like in the proof of Theorem \ref{Teorema1}, we obtain that $\{y_n\}$ is a Cauchy sequence. Since $V(M)\subseteq U(M)$ and $V(M)$ or $U(M)$ is a complete subspace of $M$, we get that there exists $w \in U(M)$ such that $\lim\limits_{n\to \infty} d(y_n, w) = 0$. So, we can find $u \in M$ such that $Uu=w$. We shall show that $Vu=w$.  
	
	Let us suppose that $Vu \neq w$. Since $y_n \neq y_{n+1}$ for every $n \in \mathbb{N}$, there exists a subsequence $\{x_{k(n)}\}$ such that $Vx_{k(n)} \neq Vu$. Thus, from (\ref{**}) we have
	\begin{equation}
		[\varphi(d(Vx_{k(n)}, Vu))]^{\eta} \leq \varphi(d(Ux_{k(n)}, Uu))
		\label{1}
	\end{equation}
	
	Since $\lim\limits_{n\to \infty} d(y_{n(k)},w) = \lim\limits_{n\to \infty} d(Ux_{k(n)},w) = 0$ and $d(Uu, w) = 0$, we have $$\lim\limits_{n\to \infty}d(Ux_{k(n)}, Uu) = 0$$ so there exists $n_0 \in \mathbb{N}$ such that for every $n \geq n_0$ we have $$d(Ux_{k(n)},Uu) \leq \dfrac{d(Vu,w)}{2}$$ so 
	\begin{equation}
		\varphi(d(Ux_{k(n)}, Uu)) \leq \varphi\left(\dfrac{d(Vu,w)}{2}\right).
		\label{2}
	\end{equation}
	
	On the other hand we have $\lim\limits_{n\to \infty} d(y_{n(k)+1},w) = \lim\limits_{n\to \infty} d(Vx_{k(n)},w) = 0$, so $$\lim\limits_{n\to \infty} d(Vx_{k(n)}, Vu) = \lim\limits_{n\to \infty} d(w,Vu) > 0.$$ Hence, there exists $n_1 \in \mathbb{N}$ such that for every $n \geq n_1$ we have $d(Vx_{n(k)}, Vu) \geq \dfrac{d(w,Vu)}{2}$ and 	
	\begin{equation}
		\varphi(d(Vx_{n(k)}, Vu)) \geq \varphi\left(\dfrac{d(w,Vu)}{2}\right).
		\label{3}
	\end{equation}
	
	Combining relations (\ref{1}), (\ref{2}) and (\ref{3}), we obtain
	$$[\varphi(d(Vx_{n(k)}, Vu))]^{\eta} \leq d(Vx_{n(k)}, Vu),$$
	which is a contradiction. Thus, $Vu=w$, and $w$ is a coincidence of $U$ and $V$ and thus we have 
	$$VUu=UVu.$$
	
	Moreover, $VVu= UVu$, which means that $Vu$ is a coincidence of $U$ and $V$.	
	
	To prove that $u$ is a common fixed point of $U$ and $V$, let us assume that $Vu \neq u$. Then, we can apply (\ref{**}) and we have
	$$\varphi(d(UVu,Uu)) \geq [\varphi(d(VVu,Vu))]^{\eta},$$
	which is a contradiction since $UVu=VVu$ and $Uu=Vu$.
	
	Thus, we have $Vu=Uu=u$, so $u$ is a common fixed point of $U$ and $V$.
	
	Let us suppose that $u$ is not the unique common fixed point of $U$ and $V$. Then, there exists $t \in M$, $t\neq u$ such that $Ut=Vt=t$. Then, we have
	$$ \varphi(d(t,u)) = \varphi(d(Ut, Uu)) \geq [\varphi(d(Vt,Vu))]^{\eta} = [\varphi(d(t,u))]^{\eta},$$
	which is a contradiction.
\end{proof}


We provide an example to illustrate our results. The example is similar to the one provided in \cite{Yesilkaya} to illustrate Theorem \ref{WronhThAB}. However, the function $\theta : (0,\infty) \to (1, \infty)$ given by $\theta(t) = e^t$ for every $t > 0$, does not belong to the class $\Theta$ (as it does not verify $\theta_3$), but is is indeed a non-decreasing function as is required in the context of the previous theorem.

\begin{example}
	The space $Y = [0,1]$ endowed with the usual metric $d(x,z) = |x-z|,$ for every $x,z \in Y$ is a complete metric space. 
	Let $U:Y \to Y$, $Ux=\dfrac x4$, for every $x \in Y$ and $V:Y\to Y$, $Vx=\dfrac x{12}$, for every $x \in Y$. We have $V(Y) \subseteq U(Y)$ and $U(Y)$ is complete. Let $\theta : (0,\infty) \to (1, \infty)$, given by  $\theta(t) = e^t$, for every $t>0$, which a non-decreasing function. Then, for every $x,z \in Y$, $x\neq z$ we have 
	$$e^\frac14|x-z| \geq e^\frac{\eta}{12}|x-z|,$$
	for $1< \eta < 3$. $U$ and $V$ are weekly compatible mappings and $0$ is the unique common fixed point.
\end{example}


\begin{thebibliography}{99}

	\bibitem{Abbas}
Abbas, M.; Nazir, T.; Radenovic, S. Common fixed points of four maps in partially ordered metric spaces.{\em Appl. Math. Lett.} \textbf{24} (2011), 1520–1526. 

\bibitem{Agarwal}
Agarwal, R.P.; El-Gebeily, M.A.; O’Regan, D. Generalized contractions in partially ordered metric spaces. {\em Appl. Anal.} \textbf{87} (2008), 109–116.

\bibitem{Daffer}
Daffer, P.Z.; Kaneko, H. On expansive mappings. {\em Math. Jpn.} \textbf{37} (1992), 733–735. 


\bibitem{Durmaz}
Durmaz, G.; Minak, G.; Altun, I. Fixed points of ordered F-contractions. {\em Hacet. J. Math. Stat.} \textbf{45} (2016), no.1, 15-21.

\bibitem{Gornicki}
G\'ornicki, J. Fixed point theorems for Kannan type mappings.{\em J. Fixed Point Theory Appl.} \textbf{19} (2017), 1–8.

\bibitem{Gurban}
Gurban, R.; Alfaqih, W.M.; Imdad, M. Fixed point results for F-expansive mappings in ordered metric spaces. {\em Commun. Fac. Sci. Univ. Ank. Ser. A1 Math. Stat.} \textbf{68} (2019), 801–808.	

\bibitem{Huang}
Huang, X.; Zhu, C.; Wen, X. Fixed point theorems for expanding mappings in partial metric spaces. {\em Analele Univ. Ovidius Constanta-Ser. Mat.} \textbf{20} (2012), 213–224. 	

\bibitem{Jleli}
Jleli, M.; Samet, B. A new generalization of the Banach contraction principle. {\em J. Inequal. Appl.} \textbf{38} (2014).

\bibitem{Jungck}
Jungck, G. Compatible mappings and common fixed points. {\em Int. J. Math. Math. Sci.} \textbf{9} (1986), 771–779.

\bibitem{Khan}
Khan, A.R., Oyetunbi, D.M.: On some mappings with a unique common fixed point. {\em J. Fixed Point Theory Appl.} \textbf{22} (2020), 1–7.

\bibitem{Kumam}
Kumam, P.; Rouzkard, F.; Imdad, M.; Gopal, D. Fixed point theorems on ordered metric spaces through a rational contraction. {\em Abstr. Appl. Anal.} (2013), 206515.

\bibitem{Minak}
Minak, G.; Altun, I. Ordered $\theta$-contractions and some fixed point results. {\em J. Nonlinear Funct. Anal.} \textbf{41} (2017. 

\bibitem{Popescu}
Popescu, O. Some remarks on the paper “Fixed point theorems for generalized contractive mappings in metric spaces”. {\em J. Fixed Point Theory Appl.} \textbf{23} (2021), 72.

\bibitem{Ran}
Ran, A.C.M.; Reurings, M.C.B. A fixed point theorem in partially ordered sets and some application to matrix equations. {\em Proc. Am. Math Soc.} \textbf{132} (2004), 1435–1443.

\bibitem{Rhoades}
Rhoades, B.E. Some fixed point theorems for pairs of mappings. {\em Math. Jnanabha} \textbf{15} (1985), 151–156.

\bibitem{Sessa}
Sessa, S. On a weak commutativity condition of mappings in fixed point considerations. {\em Pub. Inst. Math.} \textbf{32} (1982), 149–153. 

\bibitem{Taniguchi}
Taniguchi, T. Common fixed point theorems on expansion type mappings on complete metric spaces.{\em Math. Jpn.} \textbf{34} (1989), 139–142.


\bibitem{Wang}
Wang, S.Z.; Li, B.Y.; Gao, Z.M.; Iseki, K. Some fixed point theorems on expansion mappings. {\em Math. Jpn.} \textbf{29} (1984), 631–636.


\bibitem{Yesilkaya}
Yeşilkaya, S.S.; Aydin, C. Fixed Point Results of Expansive Mappings in Metric Spaces. {\em Mathematics} \textbf{8} (2020), 1800.

\end{thebibliography}
\end{document}